\documentclass[revision]{FPSAC2021}

\newtheorem{theorem}{Theorem}

\newtheorem{defn}{Definition}
\newtheorem{lem}{Lemma}
\newtheorem{proposition}{Proposition}

\newtheorem{example}{Example}

\newtheorem{remark}{Remark}

\newcommand{\al}{\alpha}
\newcommand{\be}{\beta}
\newcommand{\si}{\sigma}

\newcommand{\NN}{\mathbb{N}} 
\newcommand{\PP}{\mathbb{P}} 
\newcommand{\id}{\operatorname{id}} 
\newcommand{\kk}{\mathbf{k}} 
\newcommand{\QSym}{\operatorname{QSym}}
\newcommand{\Des}{\operatorname{Des}}
\newcommand{\Odd}{\operatorname{Odd}}
\newcommand{\Comp}{\operatorname{Comp}}
\newcommand{\Peak}{\operatorname{Peak}}

\makeatletter
\providecommand*{\shuffle}{%
  \mathbin{\mathpalette\shuffle@{}}%
}
\newcommand*{\shuffle@}[2]{%
  \sbox0{$#1\vcenter{}$}%
  \kern .15\ht0 
  \rlap{\vrule height .25\ht0 depth 0pt width 2.5\ht0}%
  \raise.1\ht0\hbox to 2.5\ht0{%
    \vrule height 1.75\ht0 depth -.1\ht0 width .17\ht0 %
    \hfill
    \vrule height 1.75\ht0 depth -.1\ht0 width .17\ht0 %
    \hfill
    \vrule height 1.75\ht0 depth -.1\ht0 width .17\ht0 %
  }%
  \kern .15\ht0 
}
\makeatother

\title[The enriched monomial basis of QSym]{Weighted posets and
the enriched monomial basis of QSym}

\author[D. Grinberg \and E.A. Vassilieva]{Darij Grinberg\thanks{\href{mailto:darijgrinberg@gmail.com}{darijgrinberg@gmail.com}}\addressmark{1}, \and Ekaterina A. Vassilieva\thanks{\href{mailto:katya@lix.polytechnique.fr}{katya@lix.polytechnique.fr}}\addressmark{2}}

\address{\addressmark{1}Department of Mathematics, Drexel University, Philadelphia, PA 19104, USA\\ \addressmark{2}Laboratoire d'Informatique de l'Ecole Polytechnique, Palaiseau, France}

\received{\today}

\abstract{Gessel's fundamental and Stembridge's peak functions are the generating functions for (enriched) $P$-partitions on labelled chains. They are also the bases of two significant subalgebras of formal power series, respectively the ring of quasisymmetric functions (QSym) and the algebra of peaks. Hsiao introduced the monomial peak functions, a basis of the algebra of peaks indexed by odd integer compositions whose relation to peak functions mimics the one between the monomial and fundamental bases of QSym. We show that the extension of monomial peaks to any composition is a new basis of QSym and generalise Hsiao's results including the product rule. To this end we introduce a weighted variant of posets and study their generating functions.}

\resume{Les fonctions de Gessel et de Stembridge sont les fonctions génératrices des $P$-partitions (enrichies) sur des chaînes étiquetées. Ce sont aussi les bases de deux sous-algèbres  des séries formelles, l'anneau des fonctions quasisymétriques (QSym) et l'algèbre des pics. Hsiao a introduit les monômes de pic, une base de l'algèbre de pics indexée par des compositions d'entier impaires dont la relation aux fonctions de pics imite celle entre les bases monomiales et fondamentales de QSym. Nous montrons que leur extension à n'importe quelle composition est une nouvelle base de QSym et généralisons les résultats de Hsiao, inclus la règle du produit. A cette fin, nous introduisons une variante pondérée des posets et étudions leurs fonctions génératrices.}


\usepackage[backend=bibtex]{biblatex}
\begin{document}
\maketitle

\section{Introduction}
\subsection{Compositions and permutation statistics}
Let $\PP=\left\{1,2,3,\dots\right\}$, $\NN = \left\{0,1,2,\ldots\right\}$, and $[n] = \left\{1,2,\dots, n\right\}$ for any $n \in \NN$. A \emph{composition} $\al = (\al_1, \al_2, \dots, \al_p)$ of an integer $n$ is a finite sequence of $\ell(\al) := p$ positive integers such that
$|\al| := \sum_i \al_i = n$. An \emph{odd composition} is a composition containing only odd integers. We let $\Comp(n)$ and $\Odd(n)$ denote the sets of compositions and odd compositions (respectively) of $n$.\\
\indent Let $S_n$ be the symmetric group on $[n]$. Given a permutation $\pi \in S_n$, we look at two significant statistics. The \emph{descent set} $\Des(\pi)$ and the \emph{peak set} $\Peak(\pi)$ of $\pi$ are the subsets of $[n-1]$ defined as 
\begin{gather*}
\Des(\pi) = \{1\leq i\leq n-1| \pi(i)>\pi(i+1)\},\\
\Peak(\pi) = \{2\leq i\leq n-1| \pi(i-1)<\pi(i)>\pi(i+1)\}.
\end{gather*}
A subset of $[n-1]$ that neither contains $1$ nor contains two consecutive integers is called \emph{peak-lacunar}.
The peak set of a permutation is peak-lacunar.

There is a natural bijection between compositions of $n$ and subsets of $[n-1]$. Namely, for a composition $\al = (\al_1, \al_2, \dots, \al_p) \in \Comp(n)$, let $\Des(\al)$ be the subset of $[n-1]$ defined as $$\Des(\al) = \{\al_1,\al_1+\al_2,\dots,\al_1+\al_2+\dots+\al_{p-1}\}.$$
As a result, $|\Comp(n)| = 2^{n-1}$. Odd compositions of $n$, on the other hand, are in one-to-one correspondence with peak-lacunar subsets of $[n-1]$. This bijection may be stated as follows (see \cite{Hsi07}). For $\al =(2i_1+1,2i_2+1,\dots,2i_p+1) \in \Odd(n)$, let 
$$
\widehat{\al} = (\overbrace{2,\dots,2}^{i_1},1,\overbrace{2,\dots,2}^{i_2},1,\dots,\overbrace{2,\dots,2}^{i_p},1).
$$ 
If
$$
\widehat{\al} = (\overbrace{1,\dots,1}^{j_1},2,\overbrace{1,\dots,1}^{j_2},2,\dots,\overbrace{1,\dots,1}^{j_l},2,1,\dots,1),
$$
then we set 
$$
\Peak(\al) = \left \{\sum_{m=1}^s (j_m+2)|1\leq s \leq l \right \} .
$$
\begin{example}
Let $\al = (1,1,3,3,1) \in \Odd(9)$. Then, $\Des(\al)= \{1,2,5,8\}$. Furthermore, one has $\widehat{\al} = (1,1,2,1,2,1,1)$ and $\Peak(\al) = \{4,7\}$.
\end{example}
Finally, a permutation $\pi \in S_n$ may be written as a word $\pi = \pi_1\pi_2\dots\pi_n$ with $\pi_i = \pi(i)$. This is called the \emph{one-line notation} of $\pi$.
Given two permutations $\pi \in S_n$ and $\si \in S_m$, let $\pi \shuffle \si$ be the subset of $S_{n+m}$ consisting of all shuffles of the words $\pi=\pi_1\pi_2\dots\pi_n$ and $n+\si=(n+\si_1) (n+\si_2)\dots (n+\si_m)$ (that is, of all permutations in $S_{n+m}$ that contain $\pi$ and $n+\si$ as subwords in their one-line notation).
\subsection{Quasisymmetric functions}
Fix a commutative ring $\kk$ and consider the ring $\kk \left[\left[ X \right]\right]$ of formal power series in countably many commuting
variables $X=\{x_{1},x_{2},x_{3},\dots\}$. In \cite{Ges84}, Gessel introduces the \emph{quasisymmetric functions}, i.e., the bounded-degree formal power series in $\kk \left[\left[ X \right]\right] $ such that for any non-negative integers $\al_1,\al_2,\dots,\al_p$ and any strictly increasing sequence of distinct indices $i_1 < i_2 < \dots < i_p$ the coefficient of $x_1^{\al_1}  x_2^{\al_2}  \cdots  x_p^{\al_p}$ is equal to the coefficient of $x_{i_1}^{\al_1}  x_{i_2}^{\al_2}  \cdots  x_{i_p}^{\alpha_p}$. The set of all quasisymmetric functions is a $\kk$-subalgebra of $\kk \left[\left[ X \right]\right]$, denoted by
$\QSym$ and called the \emph{ring of quasisymmetric functions}.
For any $n \in \NN$ and any composition $\al =(\al_1,\al_2,\dots,\al_p) \in \Comp(n)$, the \emph{monomial quasisymmetric function} $M_\al$ and the \emph{fundamental quasisymmetric function} $L_\al$ indexed by $\al$ are defined as
\begin{equation*}
\label{eq : M+L}
M_{\al} = \sum\limits_{i_1 < \cdots < i_p} x_{i_1}^{\al_1}x_{i_2}^{\al_2} \cdots x_{i_p}^{\al_p},\;\;\;\;\;\; L_{\al} = \sum\limits_{\substack{i_1 \leq \cdots \leq i_n;\\j\in \Des(\al) \Rightarrow i_j<i_{j+1}}} x_{i_1}x_{i_2} \cdots x_{i_n}.
\end{equation*}
\begin{example}
As an example, for $n=3$, we have
\begin{gather*}
M_{(2,1)}=\sum_{i<j}x_{i}^{2}x_{j}=x_{1}^{2}x_{2}+x_{1}%
^{2}x_{3}+x_{2}^{2}x_{3}+x_{1}^{2}x_{4}+x_{2}^{2}x_{4}+x_{3}^{2}x_{4}+\dots,\\
L_{(2,1)}=\sum_{i\leq j <k}x_{i}x_{j}x_k=x_{1}^{2}x_{2}+x_{1}%
^{2}x_{3}+x_{1}x_{2}x_3+x_{2}^{2}x_{3}+x_{1}^{2}x_{4}+x_{1}
x_{2}x_{4}+x_{2}^{2}x_{4}+\dots.
\end{gather*}
\end{example}
The sets $\left\{M_\al\right\}_{n\in \NN,\ \al\in \Comp(n)}$
and $\left\{L_\al\right\}_{n\in \NN,\ \al\in \Comp(n)}$
are two bases of the $\kk$-module $\QSym$.
They are related through
\begin{equation}
\label{eq : LM} L_{\al} = \sum_{{\substack{\beta \in \Comp(n);\\\Des(\alpha) \subseteq \Des(\beta)}}} M_{\beta}.
\end{equation}
\subsection{Peak and monomial peak functions}
In \cite{Ste97}, Stembridge studies another significant family of quasisymmetric functions. Given an $n \in \NN$ and an $\al \in \Odd(n)$, we define
\begin{equation*}
\label{eq : K}
K_{\al} = \sum\limits_{\substack{i_1 \leq \dots \leq i_n;\\j\in \Peak(\al) \Rightarrow  i_{j-1}<i_{j+1}}}2^{|\{i_1,i_2,\dots,i_n\}|} x_{i_1}x_{i_2} \cdots x_{i_n}.
\end{equation*}
The $K_{\al}$ are quasisymmetric functions named \emph{peak functions} because of their relation to the peak statistic on permutations (see Section \ref{sec : poset}). The set $\{K_{\al}\}_{n \in \NN,\ \al \in \Odd(n)}$ is a basis of a subalgebra of $\QSym$ named the \emph{algebra of peaks} ($\mathbf{\Pi}$ in \cite{Ste97}). Hsiao defines in \cite{Hsi07} another basis of the algebra of peaks called the \emph{monomial peak functions}. For any odd composition $\al =(\al_1,\al_2,\dots,\al_p)\in \Odd(n)$, let
\begin{equation}
\label{eq : monpeak}
\eta_{\al} = (-1)^{(n-\ell(\al))/2}\sum\limits_{\substack{i_1 \leq \dots \leq i_p}}2^{|\{i_1,i_2,\dots,i_p\}|} x_{i_1}^{\al_1}x_{i_2}^{\al_2} \cdots x_{i_p}^{\al_p}.
\end{equation}
An identity similar to Equation (\ref{eq : LM}) relates peak and monomial peak functions:
\begin{equation}
\label{eq : KE} K_{\al} = \sum_{{\substack{\beta \in \Odd(n);\\ \Peak(\beta) \subseteq \Peak(\al)}}} \eta_{\beta}.
\end{equation}

Hsiao shows that the monomial peak functions $\eta_\alpha$ form a basis of the algebra of peaks and gives closed-form formulas for their expression in terms of the classical bases of $\QSym$, their antipode, product and coproduct. His proof relies on Equation (\ref{eq : KE}) and uses the fact that the compositions are actually odd ones.\\
\indent Nevertheless, one may extend Equation (\ref{eq : monpeak}) to \textbf{any} integer composition (to the exception of the factor $(-1)^{(n-\ell(\al))/2}$ that is not well-defined if $n$ and $\ell(\al)$ are not of the same parity). Whether this extension preserves  the nice properties of the monomial peak functions appears as a natural question. We give it a \textbf{positive} answer by showing that the extended monomial peak functions form a basis of $\QSym$ and generalising the results of Hsiao. We call this new basis of $\QSym$ the \emph{enriched monomial basis}. Our methods involve various algebraic and combinatorial arguments including the introduction of a new variant of posets and $P$-partitions whose generating functions generalise the monomial, fundamental, peak and enriched monomial functions as well as other classical variants of quasisymmetric functions. Before stating these results, we recall classical results for (enriched) $P$-partitions.

Part of this material will appear in \cite{Gri20} (work in progress).

We note that the $K_\al$ basis of the algebra of peaks can also be extended to a larger family in $\QSym$ although this extension is not a basis of $\QSym$; this is done in a recent work of Khesin and Zhang \cite{KheZha20} (which also extends $\QSym$ further by introducing two extra variables $x_0$ and $x_\infty$).

\subsection{Posets and $P$-partitions}
\label{sec : poset}
One of the main motivations for studying the power series $L_\al$ and $K_\al$ is their relation to descent and peak statistics on permutations. These relations are proved through the study of (enriched) $P$-partitions. In this section, we recall the definition and main results without proofs and refer the reader to \cite{Sta01, Ges84, Ste97} for further details. 
\begin{defn}[Labelled posets] A \emph{labelled poset} $P=([n],<_P)$ is an arbitrary partial order $<_P$ on the set $[n]$. 
\end{defn} 
\begin{defn}[$P$-partition]\label{def : ppart}
Let $P = ([n],<_P)$ be a labelled poset.
A \emph{$P$-partition} is a map $f: [n]\longrightarrow \PP$ that satisfies the two following conditions:
\begin{itemize}
\item[(i)] If $i <_P j$, then $f(i) \leq f(j)$.
\item[(ii)] If $i <_P j$ and $i > j$, then $f(i) < f(j)$.
\end{itemize}
(The relations $<$ and $>$ with no subscript stand for the classical total order on $\PP$.)
\end{defn}
\begin{defn}
We let $\PP^{\pm}$ be the set $\left\{-,+\right\} \times \PP$ consisting of all pairs of a sign and a positive integer. The pair $\left(\pm, n\right) \in \PP^{\pm}$ will be denoted by $\pm n$; its \emph{absolute value} $\left|\pm n\right|$ is defined to be $n$. We equip the set $\PP^{\pm}$ with a total order given by $-1<1<-2<2<-3<\dots$. We embed $\PP$ into $\PP^{\pm}$ by identifying each $n$ with $+n$; we also let $-\PP \subseteq \PP^{\pm}$ be the set of all $-n$ for $n \in \PP$.
\end{defn}
\begin{defn}[Enriched $P$-partition]\label{def : enriched}
Let $P = ([n],<_P)$ be a labelled poset.
An \emph{enriched $P$-partition} is a map $f: [n]\longrightarrow \PP^{\pm} $ that satisfies the following two conditions:
\begin{itemize}
\item[(i)] If $i <_P j$ and $i < j$, then $f(i) < f(j)$ or $f(i) = f(j) \in \PP$.
\item[(ii)] If $i <_P j$ and $i>j$, then $f(i) < f(j)$ or $f(i) = f(j) \in -\PP$.
\end{itemize} 
\end{defn} 
Note that $P$-partitions are the same as enriched $P$-partitions with no negative values (i.e., no values of the form $-n$). A more general concept was defined in \cite{Gri18}:
\begin{defn}[$\mathcal{Z}$-enriched $P$-partition]\label{def : Zenriched}
Let $\mathcal{Z}$ be a subset of the totally ordered set $\PP^{\pm}$.
Let $P=([n],<_P)$ be a labelled poset.
A \emph{$\mathcal{Z}$-enriched $P$-partition} is
an enriched $P$-partition $f : [n] \longrightarrow \PP^{\pm} $ with $f([n]) \subseteq \mathcal{Z}$.
Let $\mathcal{L}_\mathcal{Z}(P)$ denote the set of $\mathcal{Z}$-enriched $P$-partitions.
\end{defn}
Consider the alphabet (i.e., set of indeterminates) $X = \left\{x_1,x_2,x_3,\ldots\right\}$, a labelled poset $P =([n],<_P)$, and a subset $\mathcal{Z}$ of $\PP^{\pm}$. Define the  \emph{$\mathcal{Z}$-generating function of $P$} as the formal power series
\begin{align}
\label{eq : GammaTrad}
\Gamma_{\mathcal{Z}}([n], <_P) &= \sum_{f \in \mathcal{L}_\mathcal{Z}([n],<_P)}\ \ \prod_{1\leq i \leq n}x_{|f(i)|} .
\end{align}
Given a permutation $\pi \in S_n$, let $P_\pi = ([n],<_\pi)$ denote the labelled poset on the set $[n]$, where the order relation $<_\pi$ is such that $\pi_i <_\pi \pi_j$ if and only if $i < j$ (see Figure \ref{fig : ppart}).
\begin{figure}[htbp]
\begin{center}
 \includegraphics[scale=0.20]{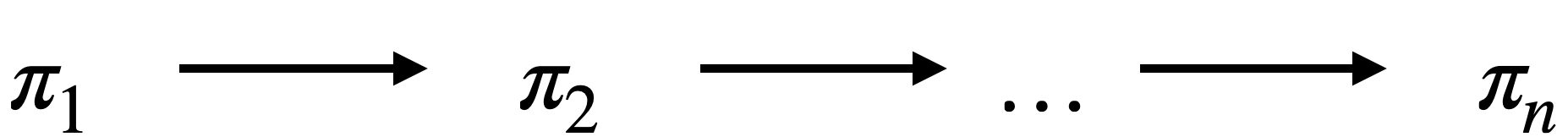}\caption{The labelled poset associated to permutation. An arrow from $\pi_i$ to $\pi_{i+1}$ means that $\pi_{i+1}$ covers $\pi_i$ in the poset.}
 \label{fig : ppart}
 \end{center}
 \end{figure}
\begin{proposition}[\cite{Ges84,Ste97}] 
\label{prop : LK}
Set $L_{\pi}= \Gamma_{\PP}([n],<_\pi)$ and $K_{\pi}= \Gamma_{\PP^\pm}([n],<_\pi)$. The function $L_\pi$ is equal to the fundamental quasisymmetric function $L_\al$ indexed by the unique composition $\al$ such that $\Des(\al) = \Des(\pi)$. Similarly, $K_\pi$ is equal to the peak function $K_\al$ indexed by the unique odd composition $\al$ such that $\Peak(\al) = \Peak(\pi)$. 
\end{proposition}
This description leads to easy product formulas for both fundamental and peak quasisymmetric functions. Indeed, one has the following proposition.
\begin{proposition}[\cite{Ges84, Ste97}]
Given two permutations $\pi \in S_n$ and $\sigma \in S_m$, the product of the generating functions $\Gamma_{\mathcal{Z}}([n],<_\pi)$ and $\Gamma_{\mathcal{Z}}([m],<_\sigma)$ is given by 
\begin{equation}
\label{eq : GGG}
\Gamma_{\mathcal{Z}}([n],<_\pi)\Gamma_{\mathcal{Z}}([m],<_\sigma) = \sum_{\gamma \in \pi \shuffle \sigma}\Gamma_{\mathcal{Z}}([n+m],<_\gamma) .
\end{equation}
\end{proposition}
Using the notation of Proposition \ref{prop : LK}, specialisations of Equation \eqref{eq : GGG} to $\mathcal{Z} \in \{\PP,\PP^\pm\}$ give the classical results 
$$
L_{\pi}L_{\sigma} = \sum_{\gamma \in \pi \shuffle \sigma}L_{\gamma}, \qquad
K_{\pi}K_{\sigma} = \sum_{\gamma \in \pi \shuffle \sigma}K_{\gamma}.
$$
\section{The enriched monomial basis of $\QSym$}
\subsection{The enriched monomial functions}
We proceed with the definition of our extension of Hsiao's monomial peak functions $\eta_\al$ to arbitrary compositions $\al$ of $n$. When $\al \in \Odd(n)$ our definition of $\eta_{\alpha}$ coincides with the one given in \cite[(6.1)]{AguBerSot06} and differs only in sign from the one given in \cite{Hsi07} and Equation \eqref{eq : monpeak}. Indeed, we remove the factor $(-1)^{(n-\ell(\al))/2}$ that is not well-defined if $\al$ is not odd.
\begin{defn}[Enriched monomials]
\label{def.etaalpha}For any $n\in\NN$ and any composition $\alpha
\in \Comp(n)$, we define a quasisymmetric function
$\eta_{\alpha}\in\QSym$ by%
\begin{equation}
\eta_{\alpha}=\sum_{\substack{\beta\in \Comp(n);\\
\Des\left(  \beta\right)  \subseteq \Des\left(  \alpha\right)  }}2^{\ell\left(
\beta\right)  }M_{\beta}. \label{eq.def.etaalpha.def}%
\end{equation}

\end{defn}

\begin{example}
\textbf{(a)} Setting $n=5$ and $\alpha=\left(  1,3,1\right)  $ in this
definition, we obtain%
\begin{align*}
\eta_{\left(  1,3,1\right)  }  &  =\sum_{\substack{\beta
\in \Comp(5);\\\Des\left(  \beta\right)  \subseteq \Des\left(
1,3,1\right)  }}2^{\ell\left(  \beta\right)  }M_{\beta}=\sum_{\substack{\beta
\in \Comp(5);\\\Des\left(  \beta\right)  \subseteq
\left\{  1,4\right\}  }}2^{\ell\left(  \beta\right)  }M_{\beta}%
\ \ \ \ \ \ \ \ \ \ \left(  \text{since }\Des\left(  1,3,1\right)  =\left\{
1,4\right\}  \right) \\
&  =2^{\ell\left(  5\right)  }M_{\left(  5\right)  }+2^{\ell\left(
1,4\right)  }M_{\left(  1,4\right)  }+2^{\ell\left(  4,1\right)  }M_{\left(
4,1\right)  }+2^{\ell\left(  1,3,1\right)  }M_{\left(  1,3,1\right)  }%
\\
&=2M_{\left(  5\right)  }+4M_{\left(  1,4\right)
}+4M_{\left(  4,1\right)  }+8M_{\left(  1,3,1\right)  }.
\end{align*}

\textbf{(b)} For any positive integer $n$, we have $\eta_{\left(  n\right)
}=2M_{\left(  n\right)  }$
(since the composition $\left(n\right)$ satisfies
$\Des\left(n\right) = \varnothing$).
Likewise, the empty composition $\varnothing
=\left(  {}\right)  $ satisfies $\eta_{\varnothing}=M_{\varnothing}$.
\end{example}

The following are easy to see from our definition of $\eta_\alpha$:

\begin{proposition}
\label{prop.eta.through-x}Let $n\in\NN$ and $\alpha\in
\Comp(n)$. Then,
\[
\eta_{\alpha}=\sum_{\substack{i_{1}\leq i_{2}\leq\cdots\leq
i_{n}  ;\\i_j=i_{j+1}\text{ for each }j\in\left[  n-1\right]
\setminus \Des\left(  \alpha\right)  }}2^{\left\vert \left\{  i_{1},i_{2}%
,\ldots,i_{n}\right\}  \right\vert }x_{i_{1}}x_{i_{2}}\cdots x_{i_{n}}.
\]
\end{proposition}

\begin{proposition}
\label{prop.eta.through-x2}Let $n \in \NN$
and $\alpha=\left(  \alpha_{1},\alpha_{2}%
,\ldots,\alpha_{p}\right)  \in \Comp(n)$. Then,
\[
\eta_{\alpha}=\sum_{i_{1}\leq i_{2}\leq\cdots\leq i_{p}}2^{\left\vert \left\{
i_{1},i_{2},\ldots,i_{p}\right\}  \right\vert }x_{i_{1}}^{\alpha_{1}}x_{i_{2}%
}^{\alpha_{2}}\cdots x_{i_{p}}^{\alpha_{p}}.
\]
\end{proposition}

We can also write the $\eta_{\alpha}$ in the fundamental basis
(generalizing \cite[Proposition 2.2]{Hsi07}):

\begin{proposition}
\label{prop.eta.through-F}Let $n$ be a positive integer. Let $\alpha
\in \Comp(n)$. Then,%
\[
\eta_{\alpha}=2\sum_{\gamma\in \Comp(n)}\left(
-1\right)  ^{\left\vert \Des\left(  \gamma\right)  \setminus \Des\left(
\alpha\right)  \right\vert }L_{\gamma}.
\]

\end{proposition}

This proposition  can be proved using the following simple identity:
\begin{lem}
\label{lem.eta.through-F.lem1}Let $S$ and $T$ be two finite sets. Then,%
\[
\sum_{I\subseteq S}\left(  -1\right)  ^{\left\vert I\setminus T\right\vert }=
\begin{cases}
2^{\left\vert S\right\vert}, & \text{if }S\subseteq T;\\
0, & \text{otherwise}.
\end{cases}
\\
\]
\end{lem}
\subsection{The $\eta_{\alpha}$ as a basis}
\begin{theorem}
\label{theorem.eta.basis}Assume that $2$ is invertible in $\kk$. Then, the
family $\left(  \eta_{\alpha}\right)  _{n \in \NN,\ \alpha\in\Comp(n)}$ is a
basis of the $\kk$-module $\QSym$.
\end{theorem}

This follows easily from the definition of $\eta_\alpha$
by a triangularity argument. Using M\"obius inversion, we can ``turn (\ref{eq.def.etaalpha.def}) around'',
obtaining an explicit expansion of the
monomial quasisymmetric functions $M_{\beta}$ in
the basis $\left(  \eta_{\alpha}\right)  _{n\in \NN, \ \alpha\in\Comp(n)}$:

\begin{proposition}
\label{prop.eta.M-through-eta}Let $n\in\NN$. Let $\beta\in
\Comp(n)$ be a composition. Then,%
\[
2^{\ell\left(  \beta\right)  }M_{\beta}=\sum_{\substack{\alpha\in
\Comp(n);\\\Des\left(  \alpha\right)  \subseteq \Des\left(
\beta\right)  }}\left(  -1\right)  ^{\ell\left(  \beta\right)  -\ell\left(
\alpha\right)  }\eta_{\alpha}.
\]

\end{proposition}

\subsection{The antipode of $\eta_{\alpha}$}
The \emph{antipode} of $\QSym$ is a certain $\kk$-linear
map $S:\QSym\rightarrow\QSym$ that is
part of the Hopf algebra structure of $\QSym$
(see \cite[Chapter 5]{GriRei20}),
but can also be defined directly.
Namely, $S : \QSym \to \QSym$ is the unique $\kk$-linear map that satisfies
\[
S\left(  M_{\alpha}\right)  =\left(  -1\right)  ^{\ell}\sum_{\substack{\gamma
\in \Comp(n);\\\Des\left(  \gamma\right)  \subseteq
\Des\left(  \alpha_{\ell},\alpha_{\ell-1},\ldots,\alpha_{1}\right)  }}M_{\gamma
}
\]
for any $n\in\NN$ and any $\alpha=\left(
\alpha_{1},\alpha_{2},\ldots,\alpha_{\ell}\right)  \in \Comp(n)$.
Also,
for each composition $\alpha$, we have $S\left(  L_{\alpha}\right)  =\left(
-1\right)  ^{\left\vert \alpha\right\vert }L_{\omega\left(  \alpha\right)  }$,
where $\omega\left(  \alpha\right)  $ is a certain composition known as the
\emph{complement} of $\alpha$. See \cite[Theorem 5.1.11 and Proposition
5.2.15]{GriRei20} for details and proofs. Note that $S$ is a
$\kk$-algebra homomorphism and an involution (that is, $S^{2} = \id$).

\begin{defn}
If $\alpha=\left(  \alpha_{1},\alpha_{2},\ldots,\alpha_{\ell}\right)  $ is a
composition, then the \emph{reversal} of $\alpha$ is defined to be the
composition $\left(  \alpha_{\ell},\alpha_{\ell-1},\ldots,\alpha_{1}\right)
$. It is denoted by $\operatorname*{rev}\alpha$.
\end{defn}

\begin{proposition}
\label{prop.eta.S}Let $n \in \NN$ and $\alpha\in\Comp(n)$. Then, the antipode $S$
of $\QSym$ satisfies
\[
S\left(  \eta_{\alpha}\right)  =\left(  -1\right)  ^{\ell\left(
\alpha\right)  }\eta_{\operatorname*{rev}\alpha}.
\]

\end{proposition}

Proposition \ref{prop.eta.S}
follows easily
from Proposition \ref{prop.eta.through-F}.
It generalizes \cite[Proposition 2.9]{Hsi07}.

\subsection{The coproduct of $\eta_{\alpha}$}

Next, consider the coproduct $\Delta:\operatorname*{QSym}\rightarrow
\operatorname*{QSym}\otimes\operatorname*{QSym}$ of the Hopf algebra
$\operatorname*{QSym}$ (see \cite[\S 5.1]{GriRei20}). We
claim the following formula for $\Delta\left(  \eta_{\alpha}\right)  $ that generalizes \cite[Cor. 2.7]{Hsi07}.

\begin{theorem}
\label{thm.Delta-eta}Let $\alpha = \left(\al_1, \al_2, \ldots, \al_p\right)$
be a composition. Then,
\[
\Delta\left(  \eta_{\alpha}\right)  =\sum_{k=0}^p
\eta_{\left(\al_1, \al_2, \ldots, \al_k\right)} \otimes
\eta_{\left(\al_{k+1}, \al_{k+2}, \ldots, \al_p\right)} .
\]

\end{theorem}

\section{The product rule for the enriched monomial basis}
To get a closed form formula for the product of two enriched monomial functions, we introduce a variant of enriched $P$-partitions with additional weights on the posets. 
\subsection{Weighted posets}
\label{sec : expposet}
\begin{defn}
A \emph{labelled weighted poset} is a triple $P = ([n], <_P, \epsilon)$ where $([n], <_P)$ is a labelled poset and $\epsilon : [n] \longrightarrow \PP$ is a map (called the \emph{weight function}).
\end{defn}
In a labelled weighted poset each node is marked with two numbers: its label $i \in [n]$ and its weight $\epsilon(i)$ (see Figure \ref{fig : poset} for an example).
\begin{figure}[htbp]
\begin{center}
 \includegraphics[scale=0.16]{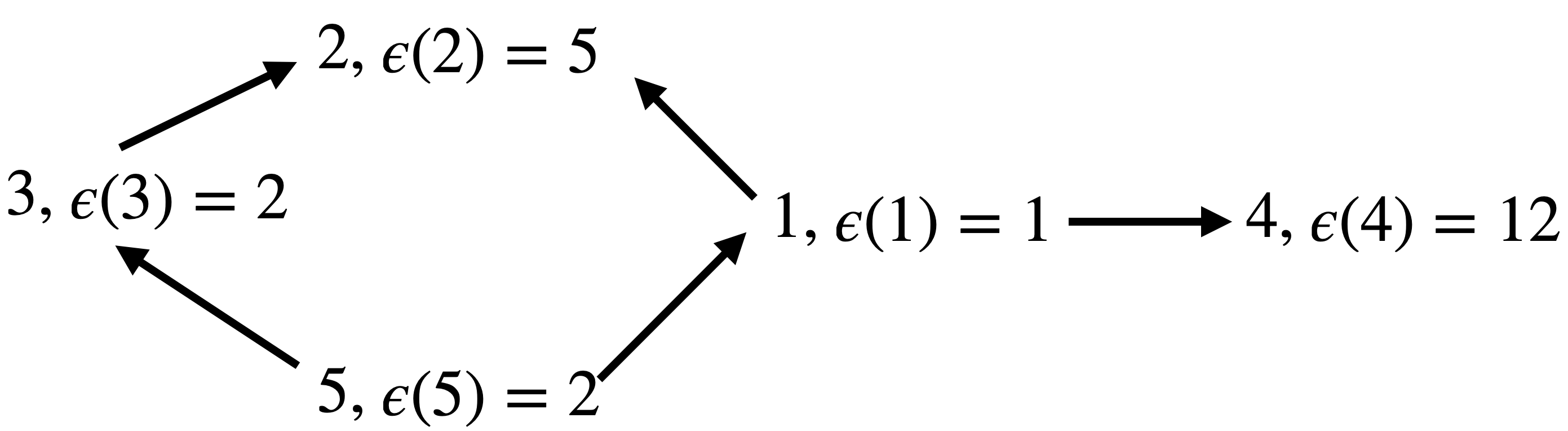}\caption{A $5$-vertex labelled weighted poset. Arrows show the covering relations.}
 \label{fig : poset}
 \end{center}
 \end{figure}
 For any set $\mathcal{Z}\subseteq \PP^{\pm}$, we define the generating function $\Gamma_\mathcal{Z}([n], <_P, \epsilon)$ of the labelled weighted poset $([n], <_P, \epsilon)$ by
\begin{equation}
\label{eq : weightGamma}
\Gamma_\mathcal{Z}([n], <_P, \epsilon) = \sum_{f \in \mathcal{L}_\mathcal{Z}([n], <_P)}\ \ \prod_{1\leq i \leq n}x^{\epsilon(i)}_{|f(i)|} .
\end{equation} 
\begin{remark}
The difference between Equations (\ref{eq : weightGamma}) and (\ref{eq : GammaTrad}) is the exponent $\epsilon(i)$ in $x^{\epsilon(i)}_{|f(i)|}$.
\end{remark}
\subsection{Universal quasisymmetric functions}
\begin{defn}
Let $\alpha = (\alpha_1, \alpha_2, \ldots, \alpha_n)$ be a composition.
Let $\pi=\pi_1\dots\pi_n$ be a permutation in $S_n$.
Let $P_{\pi,\alpha} = ([n],<_\pi,\alpha)$ denote the labelled weighted poset composed of the labelled poset $([n], <_\pi)$ and the weight function sending the vertex labelled $\pi_i$ to $\alpha_i$ (see Figure \ref{fig : monomial}).
We define the \emph{universal quasisymmetric function} $U^\mathcal{Z}_{\pi,\alpha}$ as the generating function
\begin{equation}
U^\mathcal{Z}_{\pi,\alpha} = \Gamma_\mathcal{Z}([n],<_\pi, \alpha).
\end{equation}
\end{defn}
\begin{figure}[htbp]
\begin{center}
 \includegraphics[scale=0.20]{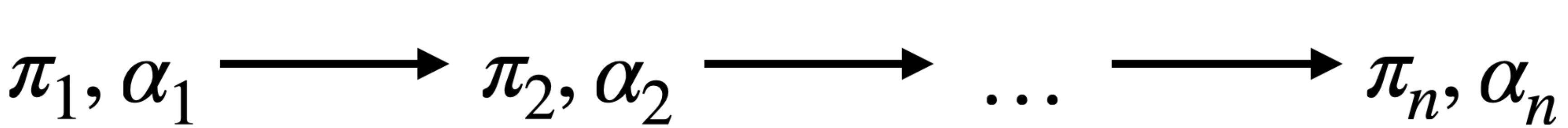}\caption{The weighted poset associated to the universal quasisymmetric function $U^\mathcal{Z}_{\pi,\alpha}$.}
 \label{fig : monomial}
 \end{center}
 \end{figure}%
\begin{lem}
\label{lem : SpecU}
Let $n \in \NN$. Let $id_{n}$ and $\overline{id_{n}}$ denote the two permutations in $S_n$ given by $id_{n} = 1~2~3\dots n$ and $\overline{id_{n}} = n~n-1~n-2\dots 1$ (in one-line notation). Denote further $(1^n)$ the composition of $n$ with $n$ entries equal to $1$. Let $\pi \in S_n$. Then,
\begin{align}
\label{eq : UE} U^\mathcal{\PP}_{\pi,(1^n)} = L_\pi, ~~~~~~U^\mathcal{\PP^\pm}_{\pi,(1^n)} = K_\pi, ~~~~~~U^{\mathcal{\PP}}_{\overline{id_{n}},\alpha} = M_{\alpha},
~~~~~U^{\mathcal{\PP^\pm}}_{id_{n},\alpha} = \eta_{\alpha}.
\end{align}
\end{lem}
As one might expect, the product of two universal quasisymmetric functions can be computed in a similar way as the product of two fundamental quasisymmetric functions. To state this result, we need to define the \emph{coshuffle} of two pairs $(\pi,\alpha)$ and $(\sigma,\beta)$.
\begin{defn}
Let $\pi \in S_n$ and $\sigma \in S_m$ be two permutations. Let $\alpha$ and $\beta$ be two compositions with respectively $n$ and $m$ entries. The \emph{coshuffle} of $(\pi,\alpha)$ and $(\sigma,\beta)$, denoted $(\pi,\alpha) \shuffle (\sigma,\beta)$, is the set of pairs $(\tau,\gamma)$ where 
\begin{itemize}
\item $\tau \in S_{n+m}$ is a shuffle of $\pi$ and $n+\sigma$ (that is, $\tau \in \pi \shuffle \sigma$), and 
\item $\gamma$ is a composition with $n+m$ entries, obtained by shuffling the entries of $\alpha$ and $\beta$ using the \emph{same shuffle} used to build $\tau$ from the letters of $\pi$ and $n+\sigma$. 
\end{itemize}
\begin{example}
$(1\mathbf{3} 2, (2, \mathbf{1}, 2))$ is a coshuffle of $(12,(2,2))$ and $(1,(1))$.
\end{example}
\end{defn}
\begin{theorem}
Let $\mathcal{Z}$ be a subset of $\PP^{\pm}$.
Let $\pi$ and $\sigma$ be two permutations in $S_n$ and $S_m$, and let $\alpha = (\al_1,\dots,\al_n)$ and $\beta = (\be_1,\dots,\be_m)$ be two compositions with $n$ and $m$ entries.
The product of two universal quasisymmetric functions is given by
\begin{equation}
\label{eq : UU}
U^\mathcal{Z}_{\pi,\alpha}U^\mathcal{Z}_{\sigma,\beta}=\sum_{(\tau,\gamma)\in(\pi,\alpha)\shuffle(\sigma,\beta)}U^\mathcal{Z}_{\tau,\gamma} .
\end{equation}
\end{theorem}
\begin{proof}[Proof (sketch).]
First notice that the product $U^\mathcal{Z}_{\pi,\alpha}U^\mathcal{Z}_{\sigma,\beta}$ is equal to the generating function of the labelled weighted poset $([n+m],<_{\pi,\sigma},(\alpha,\beta))$ depicted on the left-hand-side of Figure \ref{fig : prodmonomial} (do not pay attention to the blue arrows). 
\begin{figure}[htbp]
\begin{center}
 \includegraphics[scale=0.22]{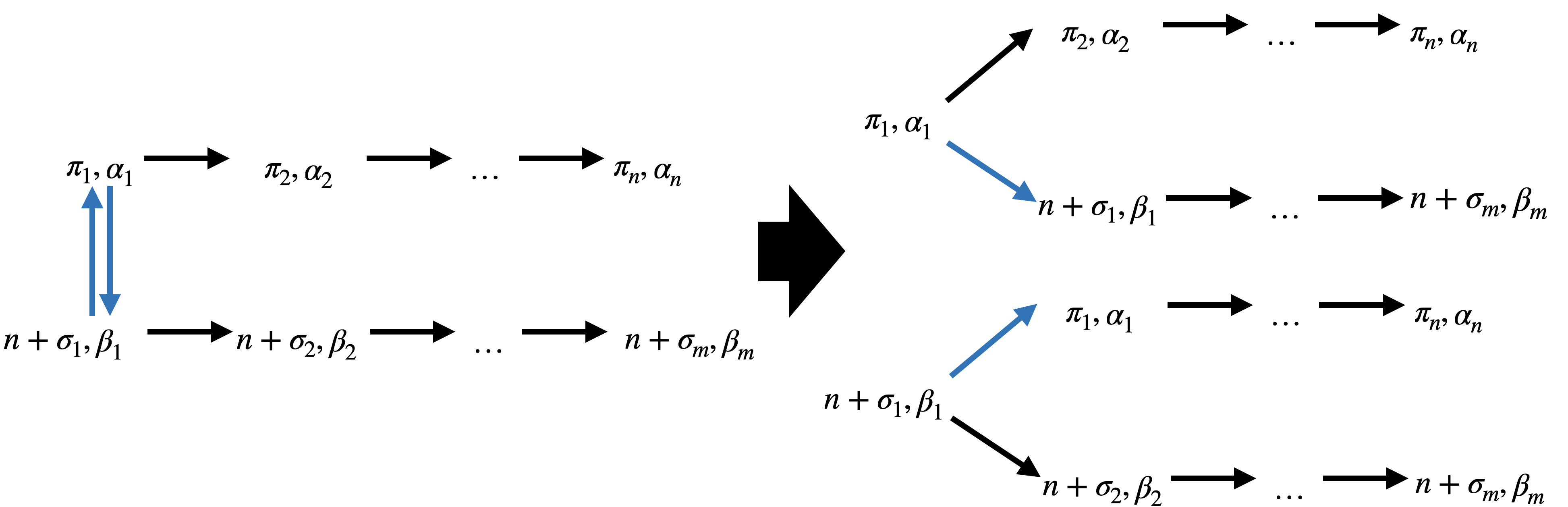}\caption{Decomposition of a double chain weighted poset into two posets with one incomparable pair less.}
 \label{fig : prodmonomial}
 \end{center}
 \end{figure}
Then the computation of $\Gamma_{\mathcal{Z}}([n+m],<_{\pi,\sigma},(\alpha,\beta))$  is performed recursively by looking at the first incomparable pair $((\pi_1,\al_1),( n+\sigma_1,\beta_1))$ (blue arrows) and splitting the sum in the generating function across the following two mutually exclusive cases:
 \begin{itemize}
\item[(i)] $f(\pi_1) < f(n+\sigma_1)$ or $f(\pi_1) = f(n+\sigma_1) \in \PP$,
\item[(ii)] $f(n+\sigma_1) < f(\pi_1)$ or $f(n+\sigma_1) = f(\pi_1) \in -\PP$.
\end{itemize}
Thus, $\Gamma_{\mathcal{Z}}([n+m],<_{\pi,\sigma},(\alpha,\beta))$ equals the sum of the generating functions of the two labelled weighted posets on the right-hand-side of Figure \ref{fig : prodmonomial}. Iterate the process until there are no more incomparable pairs to get \eqref{eq : UU}.
\end{proof}
\noindent As a final remark, we note that using sequences of $0$'s and $1$'s instead of $(1^n)$ in Lemma \ref{lem : SpecU}, our work also covers the case of weak composition quasisymmetric functions (see \cite{GuoYuZha17}).
\subsection{Product of  enriched monomials}
We compute the product of two enriched monomials. Let $\alpha$ and $\beta$ be two compositions with $n$ and $m$ entries. We have from Equations (\ref{eq : UE}) and (\ref{eq : UU}):
\begin{align}
\label{eq : MME}
\eta_\al \eta_\beta = U^{\PP^\pm}_{id_{n},\alpha}U^{\PP^\pm}_{id_{m},\beta}=\sum_{(\tau,\gamma)\in(id_{n},\alpha)\shuffle(id_{m},\beta)}U^{\mathcal{\PP^\pm}}_{\tau,\gamma} .
\end{align}
We need the following definitions.
\begin{defn}
Let $\al = (\al_1, \dots, \al_n)$ be a composition with $n$ entries. For any integer $2 \leq i \leq n-1$, we let $\al^{\downarrow\downarrow i}$ denote the following composition with $n-2$ entries:
\begin{equation*}
\al^{\downarrow\downarrow i} = (\al_1, \dots, \al_{i-2},{\al_{i-1}+\al_i+\al_{i+1}},\al_{i+2},\dots ,\al_n) .
\end{equation*}  
Furthermore, for any peak-lacunar subset $I\subseteq [n-1]$, we set $\al^{\downarrow\downarrow I}= \al$ if $I=\varnothing$ and
\[
\al^{\downarrow\downarrow I}
= \left( \left( \cdots \left( \al^{\downarrow i_k} \right) \cdots \right)^{\downarrow i_2} \right)^{\downarrow i_1} ,
\]
where $i_1, i_2, \ldots, i_k$ are the elements of $I \neq \varnothing$ in increasing order.
\end{defn}
\begin{example}
Let $\al = (2,1,4,3,2)$. We have $\al^{\downarrow\downarrow 3} = (2,8,2)$ and $\al^{\downarrow \downarrow\{2,4\}} = (12).$
\end{example}

We expand any $U^{\PP^\pm}_{\pi,\al}$ in terms of enriched monomials.

\begin{theorem}
Let $\alpha$ be a composition with $n$ entries and $\pi$ a permutation in $S_n$. We have
\begin{equation}
\label{eq : U as sum}
U^{\PP^\pm}_{\pi,\al} = \sum_{I \subseteq \Peak(\pi)}(-1)^{|I|}\eta_{\al^{\downarrow\downarrow I}} .
\end{equation}
\end{theorem}
\begin{proof}[Proof (sketch).]
From the definition of  $U^{\PP^\pm}_{\pi,\al}$, one can obtain without too much trouble that
\begin{equation*}
U^{\PP^\pm}_{\pi,\al} = \sum_{\substack{i_1\leq i_2\leq\dots \leq i_n;\\j\in \Peak(\pi) \Rightarrow \lnot \left( i_{j-1}=i_{j}=i_{j+1}\right)}}2^{|\{i_1,i_2,\dots,i_n\}|}x_{i_1}^{\al_1}x_{i_2}^{\al_2}\dots x_{i_n}^{\al_n} .
\end{equation*}
By the inclusion-exclusion principle, this can be rewritten as
\begin{equation*}
U^{\PP^\pm}_{\pi,\al} =\sum_{I \subseteq \Peak(\pi)}(-1)^{|I|} \underbrace{\sum_{\substack{i_1\leq i_2\leq\dots \leq i_n;\\j\in I \Rightarrow  i_{j-1}=i_{j}=i_{j+1}}}2^{|\{i_1,i_2,\dots,i_n\}|}x_{i_1}^{\al_1}x_{i_2}^{\al_2}\dots x_{i_n}^{\al_n}}_{\substack{= \eta_{\al^{\downarrow\downarrow I}}\\\text{(by Proposition \ref{prop.eta.through-x2})}}} .
\qedhere
\end{equation*}
\end{proof}
We may now state our final theorem: a product rule for the enriched monomial basis.
\begin{theorem}
\label{theorem : EEE}
Let $\al=(\al_1,\dots,\al_n)$ and $\beta=(\be_1,\dots,\be_m)$ be two compositions. Given a composition $\gamma$ obtained by shuffling $\al$ and $\be$ (we shall denote this by $\gamma \in \al\shuffle\be$), let $S_\beta(\gamma)$ be the set of the positions of the entries of $\beta$ in $\gamma$. Denote further $S_\beta(\gamma)-1 =\{i-1|i \in S_\beta(\gamma)\}$. Then,
\begin{equation}
\label{eq : EEE}
\eta_\al \eta_\beta
= \sum_{\substack{\gamma \in \al \shuffle \be ;\\
                   I\subseteq \left(S_{\be}(\gamma)\setminus (S_{\be}(\gamma)-1)\right) \setminus \left\{1\right\}}}
(-1)^{|I|}\eta_{\gamma^{\downarrow\downarrow I}} .
\end{equation}
The sum ranges not over compositions $\gamma$ but over ways to shuffle $\al$ with $\be$. Thus, the same $\gamma$ can appear in several addends of the sum.
\end{theorem}
\begin{proof}[Proof (sketched).]
Recall \eqref{eq : MME}, and rewrite the right-hand side using \eqref{eq : U as sum}.
Let $(\tau,\gamma)$ be a coshuffle in $(id_n,\al)\shuffle(id_m,\be)$. The index $i$ belongs to $\Peak(\tau)$ if and only if $\tau_i$ is a letter of $n+id_m$ and $\tau_{i+1}$ is a letter of $id_n$ and we have $i > 1$. That is, if and only if $i$ is the index of an entry of $\beta$ in $\gamma$ and $i+1$ is the index of an entry of $\al$.
Thus, $\Peak(\tau) = \left(S_{\be}(\gamma)\setminus (S_{\be}(\gamma)-1)\right) \setminus \left\{1\right\}$.
Therefore, we obtain \eqref{eq : EEE}.
\end{proof}
\begin{remark}
According to the proof of Theorem \ref{theorem : EEE}, the $j^{th}$ entry of $\gamma^{\downarrow \downarrow I}$ in (\ref{eq : EEE}) is a sum of $u_j$ parts of $\al$ and $v_j$ parts of $\beta$ for some integers $u_j$, $v_j$ with $\left|u_j - v_j\right| = 1$. Indeed, as $\tau$ in Equation (\ref{eq : MME}) is a shuffle of $id_n$ and $n+id_m$, ``successive'' peaks (i.e. at a distance of only two positions) in $\tau$ occur whenever $\gamma$ contains a sequence of successive parts of $\beta$ and $\alpha$ preceded by either an entry of $\alpha$ or $\beta$. As an example, if $n=3$ and $m=2$, $\alpha = (2,1,2)$ and $\beta = (1,1)$, then  $\tau = 14253$ is a shuffle of $id_3$ and $3+id_2$ whose coshuffle is $\gamma = (2,{\bf 1},1,{\bf 1},2)$. In this case $\gamma^{\downarrow \downarrow \{2,4\}} = (7)$, $\gamma^{\downarrow\downarrow 2} = (4,1,2)$ and $\gamma^{\downarrow\downarrow 4} = (2,1,4)$.
\end{remark}
\begin{example}
As an example, one has
\begin{align*}
\eta_{(1,1)}\eta_{(2,3)} &= \eta_{(1, 1, 2, 3)} + \eta_{(1, 2, 1, 3)} -\eta_{(4, 3)} + \eta_{(2,1,1,3)}+\eta_{(1,2,3,1)}\\
&\qquad -\eta_{(1,6)}+\eta_{(2,1,3,1)}-\eta_{(2,5)}+\eta_{(2,3,1,1)}-\eta_{(6,1)},\\
\eta_{(1,2)}\eta_{(2)} &=  \eta_{(2,1,2)} +   2\eta_{(1,2,2)} - \eta_{(5)}.
\end{align*}

\end{example}

\acknowledgements{DG thanks the Institut Mittag--Leffler for its hospitality in Spring 2020, and thanks Marcelo Aguiar, G\'erard H. E. Duchamp, Angela Hicks, Vasu Tewari, Alexander Zhang, and Yan Zhuang for enlightening conversations. The SageMath computer algebra system 
has been used in discovering some of the results.}

\printbibliography

\end{document}